\documentclass{amsart}

\usepackage[usenames,dvipsnames]{pstricks}
\usepackage{epsfig}
\usepackage{pst-grad} 
\usepackage{pst-plot} 

\usepackage{graphicx}

\usepackage[colorlinks=true]{hyperref}
\hypersetup{urlcolor=blue, citecolor=red}
 \newtheorem{thm}{Theorem}[section]
 \newtheorem{cor}[thm]{Corollary}
 \newtheorem{lem}[thm]{Lemma}
 
 \theoremstyle{definition}
 
 \theoremstyle{remark}

 \numberwithin{equation}{section}

\begin{document}

\title[The pressure in a deep-water Stokes wave of greatest height]
 {The pressure in a deep-water Stokes wave of greatest height}

\author[Tony Lyons]{}

\email{tlyons@ucc.ie}

\thanks{The author was supported by the Irish Research Council, Government of Ireland Postdoctoral Fellowship GOIPD/2014/34}



\date{}

\maketitle
\centerline{\scshape Tony Lyons}
\medskip
{\footnotesize
 \centerline{School of Mathematical Sciences}
   \centerline{University College Cork}
   \centerline{Cork, Ireland}
}

\begin{abstract}
In this paper we investigate the qualitative behaviour of the pressure function beneath an extreme Stokes wave over infinite depth. The presence of a stagnation point at the wave-crest of an extreme Stokes wave introduces a number of mathematical difficulties resulting in the irregularity of the free surface profile. It will be proven that the pressure decreases in the horizontal direction between a crest-line and subsequent trough-line, except along these lines themselves where the pressure is stationary with respect to the horizontal coordinate. In addition we will prove that the pressure strictly increases with depth throughout the fluid body.
\end{abstract}

\section{Introduction}\label{Sec1}
In the current paper we prove several qualitative features of the pressure function in an extreme Stokes wave over infinite depth. Extreme Stokes waves, or the Stokes wave of greatest height, are defined by the appearance of a stagnation point at the wave crest. The presence of a stagnation point at the wave-crest of an extreme Stokes wave introduces a number of mathematical difficulties resulting in the irregularity of the free surface profile \cite{Con2012, Tol1996, Con2011}. We circumvent this difficulty by excising the stagnation points and then employing a uniform limiting process. We use recent results concerning the velocity field in deep-water extreme Stokes waves which were obtained by the author in \cite{Lyo2014}, and which generalise analogous results for regular deep-water Stokes waves in \cite{Hen2011}.

A Stokes wave is a surface gravity wave, characterised by a steady, periodic, irrotational flow symmetric about the crest-line, propagating in an incompressible and inviscid fluid.  In the work \cite{Sto1880} Stokes conjectured the existence of a more exotic species of such waves, incorporating a stagnation point at the wave crest, where the fluid velocity is $(0,0)$ in the moving frame, the frame which moves with the wave speed $c>0$. The existence of the wave of greatest height as a steady, periodic two-dimensional flow in deep-water, was established in the work \cite{Tol1978} by means of sequences of regular Stokes waves tending to the extreme Stokes wave. The work of Amick et al. \cite{AFT1982} established the regularity of the free surface profile between successive crests for these extreme Stokes waves along with proving the conjecture of Stokes that these wave crests contain an angle of opening of $120^{\circ}$. The analytical approaches of these contributions  rely heavily on the irrotationality of the flow--- the incorporation of the effects of vorticity require alternative approaches and formulations cf. \cite{CS2004, CS2007, CSS2006}.

Mathematically and experimentally the properties of the pressure profile throughout a fluid body are of interest for several reasons, for example the pressure within a fluid body plays a significant role in governing particle trajectories therein, cf. \cite{Con2006,CS2010, Hen2006, Hen2008} for some recent results concerning particle trajectories in regular Stokes waves. While there are many papers which address the pressure profile beneath Stokes waves, few are concerned with the exact nonlinear theory of the flow. Nevertheless, the recent works \cite{CS2010} and \cite{Hen2011} have established the qualitative behaviour of the pressure from the fully nonlinear governing equations for regular Stokes waves in finite and infinite depth respectively.  The pressure profile for surface waves of small amplitude such as those arising as solutions of the KdV equation may be found in \cite{Joh1997}, while numerical and experimental results for nonlinear surface gravity waves are presented in \cite{EF1996} and \cite{SCJ2001} respectively.

Measurements of the fluid velocity and pressure beneath ocean surface waves are crucial to the design of cost-effective offshore structures (see \cite{DKD1992} where predictions of the forces exerted on offshore cylindrical structures are compared with those predicted by linear models of surface wind waves). In addition, given the practical applications relating the pressure measurements in water waves over finite depth to the surface profiles of these flows (see \cite{Cla2013, CC2013, Con2012a, DHOV2011, ES2008, Hen2013, OVDH2012} for some recent results this regard), the pressure profile of deep-water flows is also a practical subject worthy of further study.  Indeed, while pressure measurements are typically taken from pressure transducers which are lain along the sea-bed, this has some practical drawbacks due to the complex dynamics of the bottom boundary layer. An alternative approach which is sometimes employed is to instead use transducers suspended at a fixed-depth beneath the free-surface.

In this paper we will investigate the pressure profile of the fully nonlinear model for extreme Stokes waves over infinite depth. It will be shown that the pressure is strictly decreasing between a crest-line and subsequent trough-line within the fluid body, except along the crest-line and trough-line themselves, where the pressure is stationary with respect to the horizontal coordinate. Likewise, along the free surface it will be shown that the pressure is strictly decreasing as we move horizontally between a wave crest and subsequent wave trough, except at the crest and trough themselves, where yet again the pressure is stationary in the horizontal direction. Lastly, it will be shown that the pressure is a strictly increasing function of depth at all points throughout the fluid domain.

In Section \ref{Sec2} we present the governing equations and associated boundary conditions describing the two-dimensional fluid flow and discuss some a priori features of the velocity field and free surface profile inherent to Stokes waves.
In Section \ref{Sec3} we introduce the hodograph transform by means of a pair of conformal maps, namely the stream function and the velocity potential. While the hodograph transform is conformal almost everywhere in the fluid domain, the presence of a stagnation point at the wave crest means the map is no longer regular at that point, a difficulty that will be overcome in this paper. In Section \ref{Sec4} we present some results concerning the velocity field throughout the fluid domain, which in turn will be used in Section \ref{Sec5} to prove that the pressure decreases between any crest-line and subsequent trough-line, except along these lines themselves. Additionally it will be shown that the pressure is strictly increasing with depth at all points throughout the fluid.

\section{The Governing Equations}\label{Sec2}
Stokes waves are steady, periodic flows, which move with fixed speed, $c>0$ say, with respect to the physical frame whose coordinates we denote by $(X,Y,Z)$. Since the flow is steady, it is always possible to construct a reference frame moving relative to the physical frame in which the wave form is fixed, namely the moving frame. The points of the moving frame are labelled with coordinates $(x,y,z)$, while the transformation between the physical frame and the moving frame coordinates is given by
\begin{equation*}
x = X-ct\quad y = Y\quad z=Z.
\end{equation*}
In this frame of reference the wave profile appears fixed and consequently the governing equations for the fluid flow depend on time implicity, and are written according to
\begin{equation}\label{s2eq1}
\begin{aligned}
    &(u-c)u_x + vu_y = -P_x\\
    &(u-c)v_x+vv_y = -P_y-g
\end{aligned}
\end{equation}
in $\Omega=\left\{(x,y)\in\mathbb{R}\times(-\infty,\eta(x))\right\}$, the interior of the fluid domain. Here and throughout it is to be understood that the  $u,\ v$ and $P$ are functions of the $(x,y)$-variables. The associated boundary conditions are given by
\begin{equation}\label{s2eq2}
\begin{split}
& P=P_0\quad \text{and}\quad v=(u-c)\eta^{\prime}\quad \text{on } y=\eta\\
&(u,v)\to(0,0)\quad\text{uniformly in $x$ as $y\to-\infty$},
\end{split}
\end{equation}
where $\eta:=\eta(x)$ and the constant $P_0$ denotes the atmospheric pressure exerted on the free surface by the air column above.
Throughout this paper we consider Stokes waves as irrotational flows in an ideal fluid (incompressible, inviscid) in the absence of surface tension. As such we also require that the fluid velocity satisfy
\begin{equation}\label{s2eq3}
    u_x + v_y = 0\quad\text{and}\quad u_y - v_x = 0\quad\text{in } \Omega
\end{equation}
which ensure incompressibility and irrotationality respectively. We note that the incompressibility and irrotationality conditions simultaneously imposed on either $u$ or $v$ ensure that both components of the velocity field are harmonic in $\Omega$, that is to say
\begin{equation}
    \Delta u = \Delta v = 0 \quad \text{in } \Omega.
\end{equation}
Being periodic in the horizontal direction, Stokes waves have an associated wavelength which is defined to be the distance between two successive wave crests and which we set as $\lambda=2\pi$ without loss of generality. In the moving frame a wave-crest will be fixed at the location $x=0$ with wave-troughs at $x=\pm \pi$. Therefore the horizontal fluid velocity $u$, the surface profile $\eta$ and the pressure $P$ are all symmetric about $\left\{x=0\right\}$, while the function $\eta$ is also convex for $x\in[0,2\pi]$, cf. \cite{PT2004}.  Meanwhile the vertical fluid velocity $v$ is odd with respect to $x$. As such we need only consider the behaviour of these functions in a semi-infinite strip of width $\lambda$, namely $\Omega_\lambda = \left\{(x,y):x\in(-\pi,\pi),y\in(-\infty,\eta(x))\right\}$, and we define
\begin{equation*}
\begin{split}
&\Omega_{-} = \left\{(x,y):\ -\pi<x<0,-\infty<y<\eta(x)\right\}\\
&\Omega_{+} = \left\{(x,y):\ 0<x<\pi,-\infty<y<\eta(x)\right\}.
\end{split}
\end{equation*}
The regions $\Omega_{\pm}$ are bounded above by the free surface sections given by
\begin{equation*}
\begin{split}
&S_{-} = \left\{(x,y):\ -\pi<x<0,y=\eta(x)\right\}\\
&S_{+} = \left\{(x,y):\ 0<x<\pi,y=\eta(x)\right\}.
\end{split}
\end{equation*}
Periodicity and (anti)symmetry of the functions  $u$, $v$, $\eta$, $P$ will allow us to extend our considerations in $\overline{\Omega}_+$ throughout the fluid domain $\overline{\Omega}$.

In the case of regular, periodic, steady waves that are not near breaking, we have $u-c<0$ at the wave crest, cf. \cite{Con2011, Tol1996}. Thus for regular Stokes waves we always have
\begin{equation}\label{s2eq4.0}
    u-c<0\quad\text{on } \overline{\Omega},
\end{equation}
with the harmonic function $u$ attaining its maximum at the wave-crest. In an extreme Stokes wave, where we have $u=c$ at the wave-crest, forming a so called stagnation point (although see \cite{Con2011} for a discussion of the fluid particle behaviour at the stagnation point), relation \eqref{s2eq4.0} becomes a non-strict inequality
\begin{equation}\label{s2eq4}
u-c\leq 0\quad\text{on }\overline{\Omega},
\end{equation}
with equality achieved at the wave-cusp cf. \cite{AFT1982}.

The presence of a stagnation point on the free surface presents a number of mathematical difficulties when we try to impose maximum principles on the velocity field $(u,v)$. However a recently developed approach \cite{Con2012, Lyo2014} has allowed us to circumvent these difficulties by applying maximum principles in an excised domain. We then analyse the dynamical variables $u$ and $v$ and $P$ in a domain with the offending stagnation point removed, and obtain corresponding results for extreme Stokes waves by way of a uniform limiting process.

\section{The Hodograph Transform}\label{Sec3}
From the incompressibility of the fluid flow we may introduce a stream function $\psi$ defined by
\begin{equation}\label{s3eq1}
    \psi_y = u-c\quad\text{and}\quad\psi_x=-v.
\end{equation}
where the irrotationality of the fluid flow, cf. equation \eqref{s2eq3}, also ensures that the stream function $\psi$ is harmonic throughout the fluid domain. Integrating equations \eqref{s3eq1} and using the symmetry properties of the velocity field, the stream function is found to be periodic in the $x$-variable. Furthermore upon integrating equation \eqref{s2eq1} the Euler equation may be equivalently written as Bernoulli's Law which is given by
\begin{equation}\label{s3eq2}
\frac{(u-c)^2+v^2}{2} + gy + P = Q,
\end{equation}
where $Q$ is constant throughout the fluid domain $\overline{\Omega}$. Bernoulli's law evaluated along the free surface allows us to reformulate the free boundary problem presented in equations \eqref{s2eq1}--\eqref{s2eq3} according to
\begin{equation}\label{s3eq3}
\begin{split}
&\Delta\psi = 0\quad \text{on } (x,y)\in \Omega\\
&\vert{\nabla\psi}\vert^2 + 2gy = E\quad\text{and}\quad\psi = 0\quad \text{on } y=\eta(x)\\
&\nabla\psi\to(0,-c)\quad \text{uniformly in $x$ as $y\to-\infty$},
\end{split}
\end{equation}
where we introduce $E=Q-P_0$ which is constant throughout $\overline{\Omega}$.

In addition to the stream function, the irrotationality of the flow also ensures the existence of a velocity potential defined by
\begin{equation}\label{s3eq4}
\phi_x = u-c\quad\text{and}\quad\phi_y = v.
\end{equation}
The incompressibility condition, cf. equation \eqref{s2eq3}, ensures that the velocity potential is also harmonic in the fluid domain, while integrating \eqref{s3eq4} and imposing the symmetries of the velocity field, we deduce that the function $\phi+cx$ is $2\pi$-periodic in the horizontal direction.

Having defined a pair of harmonic conjugates, we now use $\phi$ and $\psi$  to construct the conformal map $\mathcal{H}:\Omega\to\hat{\Omega}$, where $(q,p)\in\hat{\Omega}$ is given by
\begin{equation}\label{s3eq5}
    q=-\phi(x,y)\quad\text{and}\quad p=-\psi(x,y),
\end{equation}
the so called hodograph transform \cite{CS2010}.
Under this coordinate transformation, the domain $\Omega$ whose boundary is the free surface, transforms to the domain $\hat{\Omega}$ whose boundary is simply a horizontal line, hence transforming between an unknown domain and a fixed domain. The images of the interior regions $\Omega_\pm$ are given by
\begin{equation}\label{s3eq6}
\begin{split}
    &\hat{\Omega}_- = \left\{(q,p)\in \mathbb{R}^2: q\in (-c\pi,0),p\in(-\infty,0)\right\}\\
    &\hat{\Omega}_+ = \left\{(q,p)\in \mathbb{R}^2: q\in (0,c\pi),p\in(-\infty,0)\right\}.
\end{split}
\end{equation}
Meanwhile the image of the free surface sections $S_{\pm}$ are given by
\begin{equation}\label{s3eq7}
\begin{split}
    &\hat{S}_- = \left\{q\in(-c\pi,0), p = 0\right\}\\
    &\hat{S}_+ = \left\{q\in(0,c\pi), p = 0\right\}.
\end{split}
\end{equation}
The hodograph transform between the fluid domain and the conformal domain is illustrated in Figure \ref{hodo}:
\begin{figure}[h!]
\centering
\includegraphics[trim = 4cm 12.7cm 5cm 11.5cm, clip=true, width=0.999\textwidth]{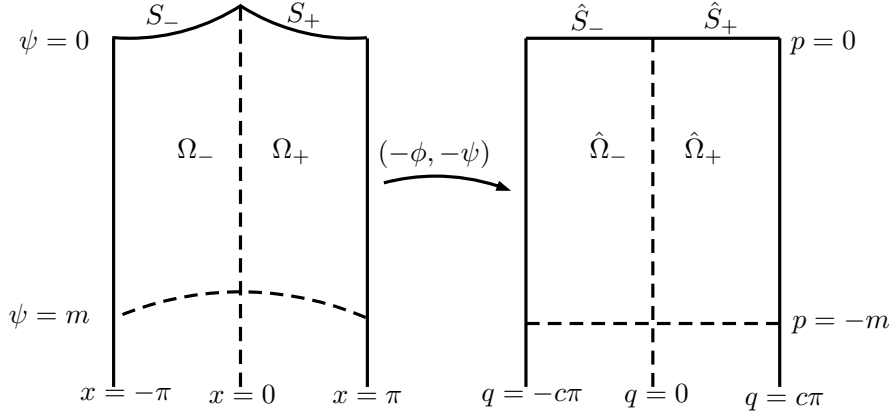}
\caption{The hodograph transform between the free boundary domain ${\overline{\Omega}}$ and the fixed domain $\overline{\hat{\Omega}}$.}
\label{hodo}
\end{figure}

We introduce the height function $h(q,p)$, such that
\begin{equation}\label{s3eq8}
    h(q,p)= y.
\end{equation}
We note that since $h$ is harmonic in the $(x,y)$-variables and the coordinate transformation \eqref{s3eq5} is conformal in the interior of $\hat{\Omega},$ then $h$ is also harmonic in the $(q,p)$-variables.
We employ $h$ and Bernoulli's Law, cf. equation \eqref{s3eq2}, to reformulate the free boundary problem given by equation \eqref{s3eq7} in the conformal domain $\hat{\Omega}$, as follows
\begin{equation}\label{s3eq9}
\begin{split}
&\Delta_{q,p}h=0\quad\text{on } \hat{\Omega}_+\\
&2(E-gh)(h_q^2+h_p^2) = 1\quad\text{on }p=0\\
&\nabla_{q,p}h\to\left(0,\frac{1}{c}\right)\quad\text{uniformly in $q$ as $p\to-\infty$}.
\end{split}
\end{equation}
The function $h$ is even and periodic with respect to $q$ with a period of $2\pi c$.
The transformation between the moving frame coordinates $(x,y)$ and the conformal coordinates $(q,p)$ is given by the following relations:
\begin{equation}\label{s3eq10}
\begin{split}
\left.
\begin{split}
&\partial_q = h_p\partial_x+h_q\partial_y\\
&\partial_p = -h_q\partial_x+h_p\partial_y
\end{split}
\right.
\qquad\text{and}\qquad
\left.
\begin{split}
&\partial_x = (c-u)\partial_q+v\partial_p\\
&\partial_y = -v\partial_q+(c-u)\partial_p
\end{split}
\right.\\
\end{split}
\end{equation}
with
\[h_q = -\frac{v}{(c-u)^2 + v^2}\qquad\text{and}\qquad h_p=\frac{c-u}{(c-u)^2+v^2}.\]

\section{The Velocity Field in the Fluid Domain}\label{Sec4}
\subsection{The Conformal Domain}
In the analysis which follows, it will be necessary to apply maximum principles in the conformal domain $\hat{\Omega}_+$. However, the fact that the closure of this domain harbours a stagnation point at $(q,p)=(0,0)$ means the velocity field does not possess the necessary regularity to impose the strong maximum principle and Hopf's maximum principle  to $u$ and $v$ in any neighbourhood of the wave crest. Nevertheless, in line with the methods developed in \cite{Con2012,Lyo2014} it is possible to apply these maximum principles in the excised conformal domain $\hat{\Omega}_+^{\varepsilon}$ illustrated in Figure \ref{excisedregion}:
\begin{figure}[h!]
\centering
\includegraphics[trim = 5cm 12.7cm 5cm 11cm, clip=true]{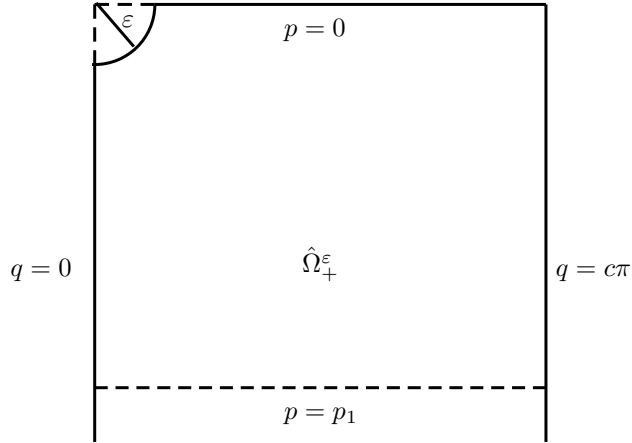}
\caption{The excised conformal domain $\hat{\Omega}_+^{\varepsilon}$ truncated by the streamline $p=p_1$, where $p_1<0$.}
\label{excisedregion}
\end{figure}

By means of the uniform limiting process outlined below, we extend these results to the region $\hat{\Omega}_+$ and hence $\Omega_+$ also.

\subsubsection{The Uniform Limiting Process}\label{Sec4.1.1}
We define $h_0:=h(q,0)$ and note that there is a sequence $\left\{h_{n}:n\geq 1\right\}$ of regular Stokes waves which converges weakly to $h_0$ in the Sobolev space $W_{\text{per}}^{1,k}(\mathbb{R})$, for $k\in(1,3)$, cf. \cite{BT2003}. Consequently this sequence of regular Stokes waves approaches the extreme wave in the H\"{o}lder space $C_{\text{per}}^{0,\alpha}(\mathbb{R})$ with $\alpha\in(0,\frac{2}{3})$, cf. \cite{Con2012}.
We introduce the variable $\xi = q+ip$ and the function
\begin{equation}
\begin{split}
    &\gamma: \hat{\Omega}_+ \to \Omega_+\\
    &\gamma:\xi \mapsto x+iy,
\end{split}
\end{equation}
which is analytic in $\hat{\Omega}_+$ and continuous on the boundary $\partial\hat{\Omega}_+$. Indeed the function $\gamma$ has an analytic continuation to any point on the boundary of $\hat{\Omega}_+$, except the stagnation point $(q,p)=(0,0)$. An application of Privalov's theorem in a semi-infinite, horizontal strip in the complex plane, as per \cite{CV2011}, then ensures the sequence $\left\{x+iy_n=\gamma_n:n\geq1\right\}$ approaches  $\gamma$ in $C_{\text{per}}(\mathbb{R}\times[-m,0])$ (the subscript per denoting periodicity with respect to $q$) cf. \cite{Con2012}, by means of the Hilbert transform of $\psi$ on this strip. To obtain the semi-infinite domain we let $m\to-\infty$. In this case the corresponding Hilbert transform may be written in terms of the usual Fourier coefficients of the periodic function $\psi$ cf. \cite{CV2011}, and it follows that $\gamma_{n}\to \gamma$ in $C_{\text{per}}^{0,\alpha}(\mathbb{R}\times(-\infty,0])$ also. Hence an infinite depth extreme wave may be obtained as a limit of regular Stokes waves. The existence of Stokes waves presenting cusps with a contained angle of 120 degrees, along with the convexity of the surface profile between successive crests, have been established by means of such sequences of regular Stokes waves tending to the wave of greatest height, see \cite{AFT1982, BT2003, PT2004} for further relevant discussions.

\subsection{Properties of the Velocity Field}
In the recent publication \cite{Lyo2014} the author established results concerning the velocity field in the free boundary domain, obtained by means of maximum principles applied in the excised conformal domain, and extended to the fluid domain by means of the uniform limiting process. We now present several results relevant to the analysis of the fluid pressure which follows in Section \ref{Sec5}.
\begin{lem}[cf. \cite{Lyo2014}]\label{s4lem1}
The velocity field $(u,v)$ is continuous throughout the closure of the fluid domain $\overline{\Omega}$, including the wave crest.
\end{lem}

\begin{lem}[cf. \cite{Lyo2014}]\label{s4lem2}
In the interior of the fluid domain $\Omega_+$ and along the free surface $S_+$ we have $v>0$, while along the crest-line and trough-line we have $v=0$.
\end{lem}

\begin{lem}[cf. \cite{Lyo2014}]\label{s4lem5}
Given the function $f(x,y) := (c-u(x,y))v(x,y)-gx$, then on the free surface we have
\begin{equation}\label{s4eq2}
    \frac{d}{dx}f(x,\eta(x)) \leq 0,\quad \text{for }x\in(0,\pi)
\end{equation}
in the case of regular deep-water Stokes waves cf. \cite{Hen2011}, and in the wave of greatest height we have the same result except it is understood that at the  wave-crest the derivative is from the right-hand side only.
\end{lem}

\begin{proof} We present a short proof of this result.
Differentiating $f(x,\eta(x))$ with respect to $x$ we find
\begin{equation}\label{s4lem5pf1}
\begin{split}
    \frac{d}{dx}\left[\left(c-u\right)v-gx\right] = v_x(1+\eta^{\prime\,2})(c-u)-g\quad\text{on }y=\eta(x).
\end{split}
\end{equation}
Differentiating the Bernoulli relation \eqref{s3eq2} evaluated along the free surface, we find
\begin{equation}\label{s4lem5pf2}
(c-u)v_y(1-\eta^{\prime\, 2})+2vv_x+g\eta^{\prime}=0\quad\text{on }y=\eta(x).
\end{equation}
With $v=(u-c)\eta^{\prime}$ on $S_+$ and dividing \eqref{s4lem5pf2} by $\eta^{\prime}$ we find
\begin{equation}\label{s4lem5pf3}
    (c-u)\left[\frac{v_y}{\eta^{\prime}}(1-\eta^{\prime\, 2})-2v_x\right] + g =0\quad\text{on }S_+.
\end{equation}
However along $S_+$  we have $\frac{d}{dx}u(x,\eta)<0$ cf. \cite{Lyo2014},  which yields
\begin{equation}\label{s4lem5pf4}
    \frac{v_y}{\eta^{\prime}}<v_x\quad\text{on }S_+,
\end{equation}
having made use of the incompressibility and irrotationality relations in equation \eqref{s2eq3}. In regular and extreme Stokes waves it is known that $\eta^{\prime\, 2}<1$, c.f. \cite{Ami1987, PT2004}, and as such we find
\begin{equation}\label{s4lem5pf5}
    v_x(1+\eta^{\prime\, 2})(c-u)-g<0\quad\text{on }S_+,
\end{equation}
with a strict sign inequality. It follows that $\frac{df}{dx}<0$ on $S_+$. However, since the left-hand side of \eqref{s4lem5pf1} is even with respect to $x$, then
$\frac{d}{dx}[(c-u)v-gx]<0$ on $S_\pm$, and it follows that $\frac{df}{dx}\leq 0$ on $y=\eta(x)$ by continuity, thus yielding inequality \eqref{s4eq2}.

\end{proof}

\begin{cor}[cf. \cite{Hen2011}]\label{s4cor1}
Since $f(x,\eta(x)) =0$ at $x=0$ for the wave of greatest height, Lemma \eqref{s4lem5} ensures that
\begin{equation}\label{s4eq3}
    f(x,\eta(x))\leq 0\quad \text{for } x\in[0,\pi].
\end{equation}
\end{cor}

\section{Main Results}\label{Sec5}
We now present the main results of this paper. In Theorem \ref{s4thm1} it will be proven that the pressure function is strictly decreasing in the horizontal direction between any crest-line and subsequent trough-line, except along these lines themselves where the pressure is stationary in the horizontal direction. In Theorem \ref{s4thm2} it will be proved that the pressure strictly increases with fluid depth throughout the fluid body.
\begin{thm}\label{s4thm1}
In the moving frame, the pressure satisfies $P_x<0$ in ${\Omega}_+$ and along $S_+$.  Along the crest-line $\left\{x=0:y\in(-\infty,\eta(0)]\right\}$ and the trough-line $\left\{x=\pi:y\in(-\infty,\eta(\pi)]\right\}$ we have $P_x=0$.
\end{thm}

\begin{proof}
The first equation in \eqref{s2eq1} along with the coordinate transformation given by equation \eqref{s3eq10}, ensure that in the moving frame the pressure satisfies
\begin{equation}\label{s4eq4}
\begin{split}
    P_x = (c-u)u_x-vu_y = \frac{u_q}{h_q^2+h_p^2},
\end{split}
\end{equation}
which is true throughout the fluid domain, except at the wave-crest where we have a stagnation point, making $u_q$ singular there. In the recent work \cite{Lyo2014} by the author, it was shown that $u_q<0$ in $\Omega_+$, and as such we conclude that $P_x<0$ therein.

Lemma \eqref{s4lem1} guarantees $f$ is continuous throughout $\overline{\Omega}$, while antisymmetry of $f$ ensures $f(0,y)=0$ for all $y\in(-\infty,\eta(0)]$. Continuity of $f$ along the boundary ensures that $f_y(0,y) = 0$ along the crest-line, and in particular at the wave-crest we observe
\begin{equation}\label{s4eq5}
    f_y(0,\eta(0)) = \lim_{y\to\eta(0)}\frac{f(0,\eta(0))-f(0,y)}{\eta(0)-y} = 0.
\end{equation}
Moreover, given $v=0$ along the crest-line it follows from equation \eqref{s2eq1} that
\begin{equation}\label{s4eq6}
    P_x(0,y)=-f_y(0,y)=0\quad\text{for }y\in(-\infty,\eta(0)],
\end{equation}
and so $P_x=0$ along the crest-line. Likewise along the trough-line we find $f(\pi,y)=-g\pi$, by antisymmetry and periodicity of $v$, and consequently we find $f_y(\pi,y)=0$ for all $y\in(-\infty,\eta(\pi)]$. Additionally, equation \eqref{s2eq1} yields
\begin{equation}\label{s4eq7}
    P_x(\pi,y) = -f_y(\pi,y)=0\quad\text{for }y\in(-\infty,\eta(\pi)],
\end{equation}
in which case we similarly deduce that $P_x=0$ along the trough-line. Along the free boundary section between the wave crest and wave trough, once again we use \eqref{s4eq4} to express
\begin{equation}\label{s4eq7a}
    P_x = \frac{u_q}{h_q^2+h_p^2}<0\quad\text{ along } S_+,
\end{equation}
having employed $u_q<0$ on $S_+$ to impose the last inequality, cf. \cite{Lyo2014}. The periodicity of $P$ now ensures that $P_x<0$ between any crest-line and subsequent trough-line within the fluid body and on its free boundary, except along these lines themselves, where we have $P_x=0$. Theorem \ref{s4thm1} is thus proved.
\end{proof}

\begin{thm}\label{s4thm2}
The pressure function increases with depth throughout the fluid body. That is to say
\begin{equation}\label{s4eq8}
    P_y<0\quad\text{throughout } \overline{\Omega}.
\end{equation}
\end{thm}

\begin{proof}
From Lemma \eqref{s4cor1} we observe that $f<0$ along $S_+$. Meanwhile the anti-symmetry and periodicity of $v$ imply $f(0,y)=0$ for $y\in(-\infty,\eta(0)]$ and $f(\pi,y)=-g\pi$ for $y\in(-\infty,\eta(\pi)]$, and so we conclude $f\leq 0$ along the boundary $\partial{\Omega}_+$. The incompressibility and irrotationality of the flow, as given by equation \eqref{s2eq3}, ensure that the function $f(x,y)$ is harmonic in the interior of the fluid domain $\Omega$, in which case $f$ must attain it maximum and minimum in $\overline{\Omega}_+$ along the boundary $\partial\Omega_+$. Maximum principles then ensure that $f<0$ on $\Omega_+$. To demonstrate this we first note that the transformation given by equation \eqref{s3eq10} is conformal within $\Omega_+$, and since $f$ is harmonic in the interior of $\Omega_+$ it follows that $f$ is also harmonic in the interior of $\hat{\Omega}_+$.

Assuming there exists an interior point $(q_0,p_0)\in\hat{\Omega}_+$ where $f(q_0,p_0)=0$, we choose $\varepsilon\in(0,\sqrt{q_0^2+p_0^2})$ and excise a quarter disc of radius $\varepsilon$ centered at $(q,p)=(0,0)$ from the conformal domain, yielding the excised conformal domain $\hat{\Omega}^{\varepsilon}_{+}$ (see Figure \ref{excisedregion}). Having eliminated the stagnation point, the function $f$ is now smooth throughout the closure of ${\hat{\Omega}}^{\varepsilon}_{+}$ and we also have the necessary regularity of the boundary to apply maximum principles. Since $(q_0,p_0)$ is an interior point of $\hat{\Omega}_{+}^{\varepsilon}$ where $f$ achieves its maximum value, then Hopf's maximum principle requires that $f\equiv0$ throughout the excised conformal domain, which is not the case. Thus $f\neq0$ in the interior $\hat{\Omega}_{+}^{\varepsilon}$. Moreover, if there exists a point $(q_0,p_0)\in\hat{\Omega}_{+}^{\varepsilon}$ such that $f(q_0,p_0)=\gamma>0$,  yet again we encounter a contradiction of Hopf's maximum principle. Owing to the boundary conditions \eqref{s2eq2}, it is always possible to choose $p_1<p_0$ such that $f(q,p_1)<\gamma$ uniformly in $q$. Considering the behaviour of $f$ on the excised conformal domain truncated by the horizontal line $p=p_1$, which we denote by $\hat{\Omega}_{+}^{\varepsilon,p_1}$(see Figure \ref{excisedregion}), Hopf's maximum principle would require $f\equiv\gamma$ throughout, since the maximum value of $f$ lies in the interior $\hat{\Omega}_{+}^{\varepsilon,p_1}$. However, it is known that $f$ is not constant. Thus we conclude $f<0$ in $\hat{\Omega}_+^{\varepsilon}$ and the uniform limiting process (cf. Section \ref{Sec4.1.1}) then ensures $f<0$ throughout the interior of $\hat{\Omega}_+$. Since the hodograph transform is conformal in the interior $\hat{\Omega}_+$, we also conclude that $f<0$ in $\Omega_+$.

Since $f<0$ in the interior and $f=0$ on the crest-line of $\Omega_+$, then along the lateral edge $\left\{x=0:y\in(-\infty,\eta(0))\right\}$, the strong maximum principle ensures
\begin{equation}\label{s4eq9}
    f_x(0,y)< 0 \qquad \text{for all }\ y\in(-\infty,\eta(0)).
\end{equation}
However along the crest-line we also obtain from equation \eqref{s2eq1}
\[P_y(0,y) = f_x(0,y)\qquad\text{for all } y\in(-\infty,\eta(0))\]
since antisymmetry ensures $v(0,y)=0$ for $y\in(-\infty,\eta(0)]$. It is immediately clear that as we descend the crest-line the pressure increases, that is to say $P_y(0,y)<0$ for all $y\in(-\infty,\eta(0))$. Antisymmetry and periodicity of $v$ ensure that along the trough-line we also have $v(\pi,y)=0$ in which case the strong maximum principle requires $v_x(\pi,y)<0$ where $y\in(-\infty,\eta(\pi))$.  Thus equation \eqref{s2eq1} yields
\begin{equation}\label{s4eq10}
    P_y(\pi,y) = (c-u(\pi,y))v_x(\pi,y)-g<0\quad \text{for } y\in(-\infty,\eta(\pi)),
\end{equation}
since relation \eqref{s2eq4} requires $c-u>0$ along the trough-line. Thus $P$ is strictly increasing as we descend the trough-line.

Next we apply maximum principles to $f_q$ in the excised region $\hat{\Omega}_+^{\varepsilon}$ and by the uniform limiting process (cf. \ref{Sec4.1.1}) we will deduce $f_q<0$ in $\Omega_+$. We note that $f$ and $f_q$ are harmonic in $\hat{\Omega}_+$. Equations \eqref{s3eq10}, \eqref{s4eq2}, \eqref{s4lem5pf1} and \eqref{s4lem5pf3} ensure
\begin{equation}\label{s4eq11}
    f_q = \frac{\frac{d}{dx}f(x,\eta(x))}{(c-u)(1+\eta^{\prime\, 2})}<0\quad \text{for } x\in(0,\pi).
\end{equation}
Consequently, we have $f_q<0$ on $\hat{S}_+$ also. Meanwhile, on the lateral edges of $\hat{\Omega}_+^{\varepsilon}$ we have $v=0$, and as such
\begin{equation}\label{s4eq12}
    f_q = v_x - \frac{g}{c-u} = \frac{P_x}{c-u}= 0
\end{equation}
along the line segments $\left\{(0,p):p\in(-\infty,-\varepsilon)\right\}$ and $\left\{(c\pi,p):p\in(-\infty,0]\right\}$,
since we previously found $P_x=0$ along the crest-line and the trough-line. Considerations similar to those which ensured $f<0$ in the interior of the excised conformal domain now ensure $f_q<0$ in $\hat{\Omega}_+^{\varepsilon}$ also, while the uniform limiting process (cf. \ref{Sec4.1.1}) extends this inequality to the domain $\hat{\Omega}_+$. Since the hodograph transform $\mathcal{H}$ is conformal in the open rectangle $\hat{\Omega}_+$, this also guarantees $f_q<0$ in $\Omega_+$ as claimed.

With $f_q<0$ in $\Omega_+$, we note that
\begin{equation*}
    f_q = \frac{(c-u)^2}{(c-u)^2+v^2}\left[v_x\left(1+\frac{v^2}{(c-u)^2}\right)-\frac{g}{c-u}\right]<0\quad\text{in } {\Omega}_+
\end{equation*}
from which we immediately deduce
\begin{equation}\label{s4eq13}
u_y < \frac{g(c-u)}{(c-u)^2+v^2}\quad\text{in }\Omega_+,
\end{equation}
having also imposed the irrotationality condition \eqref{s2eq2}.
Moreover with the conformal change of variables in equation \eqref{s3eq10} we observe
\begin{equation}\label{s4eq14}
    u_q = h_pu_x+h_qu_y=\frac{(c-u)u_x-vu_y}{(c-u)^2+v^2}<0 \quad \text{in }\Omega_+,
\end{equation}
since $u_q<0$ in $\Omega_+$, cf. \cite{Lyo2014}. Combining the relations in \eqref{s4eq13} and \eqref{s4eq14} yields the inequality
\begin{equation}
    v_y>\frac{vv_x}{c-u}\quad \text{in }\Omega_+
\end{equation}
which replaced in equation \eqref{s2eq1} ensures
\begin{equation}
    P_y = -g + (c-u)v_x-vv_y < -g + \frac{v_x}{c-u}\left[(c-u)^2+v^2\right]<0 \quad \text{in }\Omega_+.
\end{equation}
Thus $P$ is strictly decreasing a we descend into to fluid domain throughout $\Omega_+$.

Lastly equation \eqref{s2eq1} ensures
\begin{equation}
\Delta P = - 2(u_x^2+u_y^2)\leq 0 \quad \text{in } \Omega,
\end{equation}
and as such $P$ is super-harmonic in the fluid domain in which case $P$ attains its minimum on the boundary of $\Omega_+$. Furthermore the second boundary condition in \eqref{s2eq2} requires
\[\lim_{y\to-\infty}P_y = -g\quad\text{uniformly in }x.\]
Thus we find that $P$ attains its minimum on the upper boundary of $\Omega$. Moreover, since $P$ is in fact constant on the free surface Hopf's maximum principle ensures
\begin{equation}
P_y<0\quad\text{on } y=\eta(x)\text{ for all }x.
\end{equation}
Owing to the symmetry of $P$ with respect to $x$, we also deduce that $P_y<0$ in $\overline{\Omega}_-$ and by periodicity throughout the fluid body $\overline{\Omega}$, thus proving Theorem \ref{s4thm2}.
\end{proof}

\section*{Acknowledgements}
The author is grateful to the referees for several helpful comments.



\begin{thebibliography}{99}
\bibitem{Ami1987}
    C J Amick, \textit{Bounds for water waves.}
    {Arch. Rat. Mech. Anal.}, \textbf{99} (1987), 91--114.

\bibitem{AFT1982}
    C J Amick, L E Fraenkel and J F Toland, \textit{On the Stokes conjecture for the wave of extreme form.}
    Acta. Math.,  \textbf{148} (1982), 193--214.

\bibitem{BT2003}
    B Buffoni and J Toland, \textit{Analytic Theory of Global Bifurcation: An Introduction.}
    Princeton Series in Applied Mathematics, Princeton University Press, (2003).

\bibitem{Cla2013}
    D Clamond, \textit{New exact relations for easy recovery of steady wave profiles from bottom pressure measurements.}
    J. Fluid Mech., \textbf{726} (2013), 547--558.

\bibitem{CC2013}
    D Clamond and A Constantin, \textit{Recovery of steady periodic wave profiles from pressure measurements at the bed.}
    J. Fluid Mech., \textbf{714} (2013), 463-475.

\bibitem{Con2006}
    A Constantin, \textit{The trajectories of particles in Stokes waves.}
    Invent. Math., \textbf{166} (2006), 523--535.

\bibitem{Con2011}
     A Constantin, \emph{Nonlinear water waves with applications to wave-current interactions and tsunamis},
     Volume 81 of \emph{CBMS-NSF Regional Conference Series in Applied Mathematics}.
     Society for Industrial and Applied Mathematics (SIAM), Philadelphia, PA, (2011).

\bibitem{Con2012}
    A Constantin, \textit{Particle trajectories in an extreme Stokes wave}.
    IMA J. Appl. Math., \textbf{77} (2012), 293--307.

\bibitem{Con2012a}
    A Constantin, \textit{On the recovery of solitary wave profiles from pressure measurements.}
    J. Fluid Mech., \textbf{699} (2012), 376--384.

\bibitem{CSS2006}
    A Constantin, D Sattinger and W Strauss, \textit{Variational formulations for steady water waves with constant vorticity.}
    J. Fluid Mech., \textbf{548} (2006), 151-163.

\bibitem{CS2004}
    A Constantin and W Strauss, \textit{Exact steady periodc water waves with vorticity.}
    Comm. Pure Appl. Math., \textbf{57} (2004), 481--527.

\bibitem{CS2007}
    A Constantin and W Strauss, \textit{Stability properties of steady water waves with vorticity.}
    Comm. Pure Appl. Math., \textbf{60} (2007), 911--950.

\bibitem{CS2010}
    A Constantin and W Strauss, \textit{Pressure beneath a Stokes wave.}
    Comm. Pure Appl. Math., \textbf{63} (2010), 533--557.

\bibitem{CV2011}
    A Constantin and E Varvaruca, \textit{Steady periodic water waves with constant vorticity: regulairty and local bifurcation.}
    Arch. Rat. Mech. Anal., \textbf{199} (2011), 33-67.

\bibitem{DKD1992}
    W M Drennan, K K Kahma,  and M A Donelan, \textit{The velocity field beneath wind-waves - observations and inferences.}
    Coastal Eng., \textbf{18} (1992), 111--136.

\bibitem{DHOV2011}
    B Deconinck, D Henderson, K L Oliveras and V Vasan, \textit{Recovering the water-wave surface from pressure measurements.}
    Proc. 10$^{\text{th}}$ Intl. Conf. Mathematical and Numerical Analysis of Waves (WAVES 2011), 699-703.

\bibitem{ES2008}
    J Escher and T Schlurmann, \textit{On the recovery of the free surface from the the pressure within periodic travelling water waves.}
    J. Nonlin. Math. Phys., \textbf{15} (2008), 50--57.

\bibitem{EF1996}
    W A B Evans and M J Ford, \textit{An exact integral equation for solitary waves (with new numerical results for some `internal' properties).}
    Proc. Royal Soc. London Ser. A: Math Phys. Eng. Sci., \textbf{452} (1996), 373--390.

\bibitem{Hen2006}
    D Henry, \textit{The trajectories of particles in deep-water Stokes waves}.
    Int. Math, Res. Not. \textbf{2006} (2006), 23405.

\bibitem{Hen2008}
    D Henry, \textit{On the deep-water Stokes wave flow}.
    Int. Math. Res. Not., \textbf{2008} (2008), rnn 071.

\bibitem{Hen2011}
    D Henry, \textit{Presure in a deep-water Stokes wave}.
    J. Math. Fluid. Mech., \textbf{13} (2011), 251-257.

\bibitem{Hen2013}
    D Henry, \textit{On the pressure transfer function for solitary water waves with vorticity.}
    Math. Ann., \textbf{357} (2013), 23--30.

\bibitem{Joh1997}
    R S Johnson, \textit{A Modern Introduction to the Mathematical Theory of Water Waves}.
    Cambridge Texts in Applied Mathematics (Volume 19), Cambridge University Press, (1997).

\bibitem{Lyo2014}
    T Lyons, \textit{Particle trajectories in extreme Stokes waves over infinite depth}.
    Discrete Contin. Dyn. Sys. Ser. A, \textbf{34} (2014), 3095--3107.

\bibitem{OVDH2012}
    K L Oliveras, V Vasan, B Deconinck and D Henderson, \textit{Recovering surface elevation from pressure measurements}.
    SIAM J. Appl. Math., \textbf{72} (2012), 897--912.

\bibitem{PT2004}
    P I Plotnikov and J F Toland, \textit{Convexity of Stokes waves of extreme form}.
    Arch. Rat. Mech. Anal., \textbf{171} (2004), 349--416.

\bibitem{Sto1880}
    G G Stokes, \textit{Considerations relative to the greatest height of oscillatory irrotational waves which can be propagated without change in form}.
    Math. Phys. Papers, \textbf{1} (1880), 225-228.

\bibitem{SCJ2001}
    C Swan, I P Cummins and R L James, \textit{An experimental study of two-dimensional surface water waves propagating on depth-varying currents. Part 1: Regular waves}.
    J. Fluid. Mech., \textbf{428} (2001), 273-304.

\bibitem{Tol1978}
    J F Toland, \textit{On the existence of a wave of greatest height and Stokes conjecture}.
    Proc. Royal Soc. London Ser. A: Math., Phys., Eng., Sci., \textbf{363} (1978), 469--485.

\bibitem{Tol1996}
    J F Toland, \textit{Stokes waves}.
    Topol. Meth. Nonlin. Anal., \textbf{7} (1996), 1--48.

\end{thebibliography}
\end{document}